\documentclass[10pt]{amsart}
\usepackage{amsmath, amsfonts, amscd, latexsym, amsthm, amssymb}
\usepackage{hyperref}
\usepackage{color}

\def \z{\mathbb{Z}}
\def \r{\mathbb{R}}
\def \n{\mathbb{N}}

\def \C{\mathcal{C}}

\def \I{\mathcal{I}}
\def \J{\mathcal{J}}
\def \R{\mathbb{R}_{\geq 0}^n}

\def \M{\mathcal{M}}
\def \S{\mathcal{S}}

\def \m{\mathfrak{m}}
\def \a{\mathfrak{a}}

\def \k{{\bf k}}

\def \b{\mathfrak{b}}

\def \.{\cdot}

\def \vol{\textup{vol}}
\def \covol{\textup{covol}}

\def \In{{\bf in}}

\def \conv{\textup{conv}}

\theoremstyle{plain}
\newtheorem{Th}{Theorem}[section]

\newtheorem{Prop}[Th]{Proposition}
\newtheorem{Cor}[Th]{Corollary}

\theoremstyle{definition}
\newtheorem{Ex}[Th]{Example}
\newtheorem{Def}[Th]{Definition}
\newtheorem{Rem}[Th]{Remark}

\begin{document}
\title{Convex bodies and multiplicities of ideals}

\author{Kiumars Kaveh}
\address{Department of Mathematics, School of Arts and Sciences, University of Pittsburgh, 
301 Thackeray Hall, Pittsburgh, PA  15260, U.S.A.}
\email{kaveh@pitt.edu} 

\thanks{The first author is partially supported by a
Simons Foundation Collaboration Grants for Mathematicians (Grant ID: 210099) and a National Science Foundation
(Grant ID: 1200581).}

\thanks{The second author is partially supported by the Canadian Grant N 156833-12.}

\author{A. G. Khovanskii}
\address{Department of Mathematics, University of Toronto, Toronto, Canada; 
Moscow Independent University; Institute for Systems Analysis, Russian Academy of Sciences}
\email{askold@math.utoronto.ca}

\dedicatory{Dedicated to Viktor Matveyevich Buchstaber for the occasion of his 70th birthday}

\keywords{multiplicity, local ring, Hilbert-Samuel function, convex body} 
\subjclass[2010]{Primary: 13H15, 11H06; Secondary: 13P10}

\date{\today}

\maketitle

\begin{abstract}
We associate {convex regions} in $\r^n$ to $\m$-primary graded sequences of subspaces, in particular $\m$-primary graded sequences of ideals, in a large class of local algebras (including analytically irreducible local domains).
These {convex regions} encode information about Samuel multiplicities. This is in the spirit of the theory of Gr\"obner bases and Newton polyhedra on one hand, and the theory of Newton-Okounkov bodies for linear systems on the other hand. We use this to give a new proof, as well as a generalization of a Brunn-Minkowski inequality for multiplicities due to Teissier and Rees-Sharp. 
\end{abstract}

\section*{Introduction}
The purpose of this note is to employ, in the local case, techniques from the theory of semigroups of integral points and Newton-Okounkov bodies (for the global case) and to obtain new results as well as new proofs of some previously known results about multiplicities of ideals in local rings. 

Let $R = \mathcal{O}_{X, p}$ be the local ring of a point $p$ on an $n$-dimensional irreducible algebraic variety $X$ over an algebraically closed field $\k$. 
Let $\m$ denote the maximal ideal of $R$ and let $\a$ be an $\m$-primary ideal, i.e. $\a$ is an ideal containing a power of the maximal ideal $\m$. Geometrically speaking, $\a$ is $\m$-primary if its zero set (around $p$) is the single point $p$ itself. Let $f_1, \ldots, f_n$ be $n$ generic elements in $\a$. The {\it multiplicity} $e(\a)$ of the ideal $\a$ is the intersection multiplicity, at the origin, of the hypersurfaces $H_i = \{ x \mid f(x) = 0\}$, $i=1, \ldots, n$ (it can be shown that this number is independent of the choice of the $f_i$). According to Hilbert-Samuel's theorem, the multiplicity $e(\a)$ is equal to: 
$$n!~\lim_{k \to \infty} \frac{\dim_\k(R/\a^k)}{k^n}.$$
(This result is analogous to Hilbert's theorem on the Hilbert function and degree of a projective variety.)
More generally, let $R$ be an $n$-dimensional Noetherian local domain over $\k$ (where $\k$ is isomorphic to 
the residue field $R/\m$ and $\m$ is the maximal ideal). 
Let $\a$ be an $\m$-primary ideal of $R$. Since $\a$ contains a power of the maximal ideal $\m$, $R/\a$ is finite dimensional regarded as a vector space over $\k$.  
The {\it Hilbert-Samuel function} of the $\m$-primary ideal $\a$ is defined by: $$H_\a(k) = 
\dim_\k(R/\a^k).$$ For large values of $k$, $H_\a(k)$ coincides with a polynomial of degree $n$ called 
the {\it Hilbert-Samuel polynomial} of $\a$. The {\it Samuel multiplicity}, $e(\a)$ of $\a$ is defined to be the leading coefficient of $H_\a(k)$ multiplied by $n!$. 

It is well-known that the Samuel multiplicity satisfies a Brunn-Minkowski inequality \cite{Teissier1, RS}. That is,
for any two $\m$-primary ideals $\a$, $\b \in R$ we have:
\begin{equation} \label{equ-Brunn-Mink-intro} 
e(\a)^{1/n} + e(\b)^{1/n} \geq e(\a\b)^{1/n}.
\end{equation}

More generally we define multiplicity for {\it $\m$-primary graded sequences of subspaces}. That is, a sequence
$\a_1, \a_2, \ldots$ of $\k$-subspaces in $R$ such that for all $k, m$ we have $\a_k \a_m \subset \a_{k+m}$, 
and { $\a_1$ contains a power of the maximal ideal $\m$} (Definition \ref{def-graded-seq-subspace}). 
We recall that if $\a, \b$ are two $\k$-subspaces of $R$, 
$\a\b$ denotes the $\k$-span of all the $xy$ where $x \in \a$ and $y \in \b$. In particular, a graded sequence $\a_\bullet$ where each $\a_k$ is an 
$\m$-primary ideal, is an $\m$-primary graded sequence of subspaces. We call such $\a_\bullet$ an 
{\it $\m$-primary graded sequence of ideals}.

For an $\m$-primary graded sequence of subspaces we define multiplicity $e(\a_\bullet)$ to be: 
\begin{equation} \label{equ-multi-graded-seq-ideals-intro}
e(\a_\bullet) = n!~\lim_{k \to \infty} \frac{\dim_\k(R/\a_k)}{k^n}.
\end{equation} 
(It is not a priori clear that the limit exists.)

We will use convex geometric arguments to prove the existence of the limit in \eqref{equ-multi-graded-seq-ideals-intro}  
and a generalization of \eqref{equ-Brunn-Mink-intro} to $\m$-primary graded sequences of subspaces, 
for a large class of local domains $R$. 



{Let us briefly discuss the convex geometry part of the story}. Let $\C$ be a closed strongly convex cone with apex at the origin 
(i.e. $\C$ is a convex cone and does not contain any line). 
We call a closed convex set $\Gamma \subset \C$, a {\it $\C$-convex 
region} if for any $x \in \Gamma$ we have $x + \C \subset \Gamma$. We say that $\Gamma$ is {\it cobounded} if $\C \setminus \Gamma$ is bounded. It is easy to verify that
the set of cobounded $\C$-convex regions is closed under addition (Minkowski sum of convex sets) and multiplication with a positive real number. For a cobounded $\C$-convex region $\Gamma$ we call the volume of the bounded region 
$\C \setminus \Gamma$ the {\it covolume of $\Gamma$} and denote it by $\covol(\Gamma)$. Also we refer to $\C \setminus \Gamma$ 
as a {\it $\C$-coconvex body}. (Instead of working with convex regions one can alternatively work with coconvex bodies.) 
In \cite{Askold-Vladlen, Askold-Vladlen2}, similar to convex bodies and their volumes (and mixed volumes), the authors 
develop a theory of convex regions and their covolumes (and mixed covolumes). Moreover they prove an analogue of 
the Alexandrov-Fenchel inequality for mixed covolumes (see Theorem \ref{thm-alexandrov-fenchel-covolume}). 
The usual Alexandrov-Fenchel inequality is an important inequality about mixed volumes of convex bodies in $\r^n$ and generalizes the classical isoperimetric inequality and the Brunn-Minkowski inequality. In a similar way, the result in \cite{Askold-Vladlen} 
implies a Brunn-Minkowski inequality for covolumes, that is, for any two cobounded $\C$-convex regions $\Gamma_1$, $\Gamma_2$  where $\C$ is an $n$-dimensional cone, we have:
\begin{equation} \label{equ-intro-BM-inequ-covol}
\covol(\Gamma_1)^{1/n} + \covol(\Gamma_2)^{1/n} \geq 
\covol(\Gamma_1 + \Gamma_2)^{1/n}.
\end{equation}
We associate convex regions to $\m$-primary graded sequences of subspaces (in particular $\m$-primary ideals) 
and use the inequality \eqref{equ-intro-BM-inequ-covol} to prove the Brunn-Minkowski inequality for multiplicities.
To associate a convex region to a graded sequence of subspaces we need a valuation on the ring $R$. 
We will assume that there is a valuation $v$ on $R$ with values in $\z^n$ (with respect to a total order on $\z^n$ respecting addition) such that 
the residue field of $v$ is $\k$, and moreover the following conditions (i)-(ii) hold \footnote{In \cite{KKh-Annals} a valuation $v$ with values 
in $\z^n$ and residue field $\k$ is called a {\it valuation with one-dimensional leaves (see Definition \ref{def-valuation}).}}. 
We call such $v$ a {\it good valuation} on $R$ (Definition \ref{def-good-valuation}):

{
\noindent (i) Let $\S = v(R \setminus \{0\}) \cup \{0\}$ be the value semigroup of $(R, v)$. Let $\C = C(\S)$ be the closure of the convex hull of $\S$. It is a closed convex cone with apex at the origin. We assume that $\C$ is a strongly convex cone.
 
Let $\ell: \r^n \to \r$ be a linear function. For any $a \in \r$ let $\ell_{\geq a}$ denote the half-space $\{ x \mid \ell(x) \geq a\}$. 
Since the cone $\C = C(\S)$ associated to the semigroup $\S$ is assumed to be strongly convex we can find a linear function $\ell$ 
such that the cone $\C$ lies in $\ell_{\geq 0}$ and it intersects the hyperplane $\ell^{-1}(0)$ only at the origin. 

\noindent (ii) We assume there exists $r_0 > 0$ and a linear function $\ell$ as above such that for any $f \in R$, 
if $\ell(v(f)) \geq kr_0$ for some $k>0$ then $f \in \m^k$.
 
Let $\M_k = v(\m^k \setminus \{0\})$ denote the image of $\m^k$ under the valuation $v$. 
The condition (ii) in particular implies that for any $k>0$ we have $\M_k \cap \ell_{\geq kr_0} = \S \cap \ell_{\geq kr_0}$.
}

As an example, let $R=\k[x_1, \ldots, x_n]_{(0)}$ be the algebra of polynomials localized at the maximal ideal $(x_1, \ldots, x_n)$. Then the map $v$ which associates to a polynomial its lowest exponent (with respect to some term order) defines a {good valuation} on $R$ and the value semigroup $\S$ coincides with the semigroup $\z_{\geq 0}^n$, that is, the semigroup of all the integral points in the 
positive orthant $\C = \r_{\geq 0}^n$. In the same fashion any regular local ring has a good valuation, as well as the local ring of a toroidal singularity (Example \ref{ex-good-val-toric-local-ring} and Theorem \ref{th-good-val-reg-local-ring}). More generally, in Section \ref{sec-valuation-ideal} we see that an analytically irreducible local domain $R$ has a good valuation (Theorem \ref{th-good-val-S/R} and Theorem \ref{th-good-val-analytically-irr-sing}; see also \cite[Theorem 4.2 and Lemma 4.3]{Cutkosky1}). A local ring $R$ is said to be analytically irreducible if its completion is an integral domain. Regular local rings and local rings of toroidal singularities are analytically irreducible. (We should point out that in the first version of the paper we had addressed only the case where 
$R$ is a regular local ring or the local ring of a toroidal singularity.)

{Given a good $\z^n$-valued valuation $v$ on the domain $R$, we associate the (strongly) convex cone $\C = \C(R) \subset \r^n$ to the domain $R$ which is the closure of convex hull of the value semigroup $\S$. 
Then to each $\m$-primary graded sequence of subspaces $\a_\bullet$ in $R$
we associate a convex region $\Gamma(\a_\bullet) \subset \C$, such that 
the set $\C \setminus \Gamma(\a_\bullet)$ is bounded (Definition \ref{def-Gamma-I}).}
The main result of the manuscript (Theorem \ref{th-multi-ideal-covol}) is that 
the limit in (\ref{equ-multi-graded-seq-ideals-intro}) exists and: 
\begin{equation} \label{equ-intro-main}
e(\a_\bullet) = n!~\covol(\Gamma(\a_\bullet)).
\end{equation}
{The equality \eqref{equ-intro-main} and the Brunn-Minkowski inequality for covolumes (see \eqref{equ-intro-BM-inequ-covol} or 
Corollary \ref{cor-Brunn-Mink-covol}) 
are the main ingredients in proving a generalization of the inequality \eqref{equ-Brunn-Mink-intro} to $\m$-primary graded sequences of subspaces (Corollary \ref{cor-Brunn-Mink-multi}).}

{ We would like to point out that the construction of $\Gamma(\a_\bullet)$ is an analogue of the construction of the Newton-Okounkov body of a linear system on an algebraic variety (see \cite{Okounkov-log-concave}, \cite{Okounkov-Brunn-Mink}, \cite{KKh-Annals}, \cite{LM}). In fact, the approach and results in the present paper are analogous to the approach and results in 
\cite{KKh-Annals} regarding the asymptotic behavior of Hilbert functions of a general class of graded algebras.
In the present manuscript we also deal with certain graded algebras (i.e. $\m$-primary graded sequences of subspaces) 
but instead of dimension of graded pieces we are interested in the codimension (i.e. dimension of $R / \a_k$), that is why in our main theorem (Theorem \ref{th-multi-ideal-covol}) the covolume of a convex region appears as opposed to the volume of a convex body (\cite[Theorem 2.31]{KKh-Annals}). Also our Theorem \ref{th-multi-ideal-covol} generalizes 
\cite[Corollary 3.2]{KKh-Annals} which gives a formula for the degree of a projective variety $X$ in terms of the volume of its corresponding Newton-Okounkov body, because the Hilbert function of a projective variety $X$ can be regarded as the difference derivative of the 
Hilbert-Samuel function of the affine cone over $X$ at the origin and hence has the same leading coefficient.}

On the other hand, the construction of $\Gamma(\a)$ generalizes the notion of the Newton diagram of a power series
(see \cite{Kushnirenko} and \cite[Section 12.7]{AVG}). To a monomial ideal in a polynomial ring (or a power series ring), 
i.e. an ideal generated by monomials, one can associate its (unbounded) {\it Newton polyhedron}. It is the convex hull of the exponents of the monomials appearing in the ideal. The {\it Newton diagram} of a monomial ideal is the union of the bounded faces of the Newton polyhedron. One can see that for a monomial ideal $\a$, the convex region $\Gamma(\a)$ coincides with its Newton polyhedron (Theorem \ref{th-covol-monomial}).
The main theorem in this manuscript (Theorem \ref{th-multi-ideal-covol}) for the case of monomial ideals recovers the local case of the well-known theorem of Bernstein-Kushnirenko, about computing the multiplicity at the origin of a system $f_1(x) = \cdots = f_n(x) = 0$ where the $f_i$ are generic functions from $\m$-primary monomial ideals (see Section \ref{sec-monomial-ideal} and  \cite[Section 12.7]{AVG}).


{ Another immediate corollary of \eqref{equ-intro-main} is the following: let $\a$ be an $\m$-primary ideal in 
$R = \k[x_1, \ldots, x_n]_{(0)}$. Fix a term order on $\z^n$ and for each $k > 0$ let $\In(\a^k)$ denote the initial ideal of the ideal $\a^k$
(generated by the lowest terms of elements of $\a^k$). Then the sequence of numbers $$\frac{e(\In(\a^k))}{k^n}$$ is decreasing and converges to 
$e(\a)$ as $k \to \infty$ (Corollary \ref{cor-Lech}). }


The Brunn-Minkowski inequality proved in this paper { is closely related to} the more general Alexandrov-Fenchel inequality for mixed multiplicities. Take $\m$-primary ideals 
$\a_1, \ldots, \a_n$ in a local ring $R = \mathcal{O}_{X, p}$ of a point $p$ on an $n$-dimensional 
algebraic variety $X$. The {\it mixed multiplicity} $e(\a_1, \ldots, \a_n)$ is equal to the intersection multiplicity, at the origin, of the hypersurfaces 
$H_i = \{ x \mid f_i(x) = 0\}$, $i = 1, \ldots, n$, where each $f_i$ is a generic function from $\a_i$. 
Alternatively one can define the mixed multiplicity as the polarization of the Hilbert-Samuel multiplicity 
$e(\a)$, i.e. it is the unique function $e(\a_1, \ldots, \a_n)$ which is invariant under permuting the arguments, is multi-additive with respect to product of ideals, and for any $\m$-primary ideal 
$\a$ the mixed multiplicity $e(\a, \ldots, \a)$ coincides with $e(\a)$. In fact, in the above the $\a_i$ need not be ideals and it suffices for them to be $\m$-primary subspaces.   

The Alexandrov-Fenchel inequality is the following inequality among the mixed multiplicities of the $\a_i$:
\begin{equation} \label{equ-Alex-Fenchel-mixed-multi}
e(\a_1, \a_1, \a_3, \ldots, \a_n) e(\a_2, \a_2, \a_3, \ldots, \a_n) \geq e(\a_1, \a_2, \a_3, \ldots, \a_n)^2 
\end{equation}

{When $n = \dim R = 2$ it is easy to see that the Brunn-Minkowski inequality \eqref{equ-Brunn-Mink-intro} and the 
Alexandrov-Fenchel inequality \eqref{equ-Alex-Fenchel-mixed-multi} are equivalent.
By a reduction of dimension theorem for mixed multiplicities one can get a proof of the Alexandrov-Fenchel inequality \eqref{equ-Alex-Fenchel-mixed-multi} from the Brunn-Minkowski inequality \eqref{equ-Brunn-Mink-intro} for $\dim(R)=2$. 
The Brunn-Minkowski inequality \eqref{equ-Brunn-Mink-intro} was originally proved in \cite{Teissier1, RS}.}

In a recent paper \cite{KKh-mixed-multi} we give a simple proof of the Alexandrov-Fenchel inequality for mixed multiplicities of ideals using arguments similar to but different from those of this paper. This then implies an Alexandrov-Fenchel inequality for covolumes of convex regions 
(in a similar way that in \cite{KKh-Annals} and in \cite{Askold-BZ} the authors obtain an alternative proof of the usual Alexandrov-Fenchel inequality for volumes of convex bodies from similar inequalities for intersection numbers of divisors on algebraic varieties). 

We would like to point out that the Alexandrov-Fenchel inequality in \linebreak \cite{Askold-Vladlen} 
for covolumes of coconvex bodies is related to an analogue of 
this inequality for convex bodies in higher dimensional hyperbolic space (or higher dimensional Minkowski space-time).
{ From this point of view, the Alexandrov-Fenchel inequality has been proved for certain coconvex bodies in \cite{Fillastre}.}

After the first version of this note was completed we learned about the recent papers \cite{Cutkosky1, Cutkosky2} and 
\cite{Fulger} which establish the existence of limit (\ref{equ-multi-graded-seq-ideals-intro}) in more general settings.
We would also like to mention the paper of Teissier \cite{Teissier2} which discusses Newton polyhedron of a power series, and notes the relationship/analogy between notions from local commutative algebra and convex geometry. { Also we were notified that, for ideals in a polynomial ring, ideas similar to construction of $\Gamma(\a_\bullet)$ (see Definition \ref{def-Gamma-I}) 
appears in \cite{Mustata} were the highest term of polynomials is used instead of a valuation. Moreover, in \cite[Corollary 1.9]{Mustata} the Brunn-Minkowski-inequality for multiplicities of graded sequences of $\m$-primary ideals is proved for regular local rings using Teissier's Brunn-Minkowski \eqref{equ-Brunn-Mink-intro}.}

{Finally as the final version of this manuscript was being prepared for publication, 
the preprint of D. Cutksoky \cite{Cutkosky3} appeared in arXiv.org in which the author uses similar methods to 
prove Brunn-Minkowski inequality for graded sequences of $\m$-primary ideals in local domains.}

And few words about the organization of the paper: Section \ref{sec-mixed-vol} recalls basic background material about 
volumes/mixed volumes of convex bodies. Section \ref{sec-convex-diag} is about convex regions and their covolumes/mixed covolumes, which we can think of as a local version of the theory of mixed volumes of convex bodies. In Sections \ref{sec-semigp-int} and \ref{sec-semigp-ideal} we associate a convex region to a primary sequence of subsets in a semigroup and prove the main combinatorial result required later (Definition \ref{def-Gamma-I-semigroup} and Theorem \ref{th-vol-Gamma-semigroup}). { In Section \ref{sec-multi-ideal} we recall some basic definitions and facts from commutative algebra about multiplicities of $\m$-primary 
ideals (and subspaces) 
in local rings. The next section (Section \ref{sec-monomial-ideal}) discusses the case of monomial ideals and the Bernstein-Kushnirenko theorem.
Finally in Section \ref{sec-valuation-ideal}, using a valuation on the ring $R$, we associate a convex region $\Gamma(\a_\bullet)$
to an $\m$-primary graded sequence of subspaces $\a_\bullet$ and prove the main results of this note (Theorem \ref{th-multi-ideal-covol} and Corollary \ref{cor-Brunn-Mink-multi}).}\\

{\bf Acknowledgement.} The first author would like to thank Dale Cutkosky, Vladlen Timorin and Javid Validashti for helpful discussions. We are also thankful to Bernard Teissier, Dale Cutkosky, Francois Fillastre and Mircea Musta\c{t}\u{a} for informing us about their interesting papers \cite{Teissier2}, \cite{Cutkosky1, Cutkosky2}, \cite{Fillastre} and \cite{Mustata}.

\section{Mixed volume of convex bodies} \label{sec-mixed-vol}
The collection of all convex bodies in $\r^n$ is a cone, that is, we can add convex bodies and multiply a convex body with a
positive number. For two convex bodies $\Delta_1, \Delta_2 \subset \r^n$, their (Minkowski) sum $\Delta_1 + \Delta_2$ is  
$\{ x + y \mid x \in \Delta_1, y \in \Delta_2\}$. Let $\vol$ denote the $n$-dimensional volume in $\r^n$ with respect to the
standard Euclidean metric. The function $\vol$ is a homogeneous polynomial of degree $n$ on the cone of convex bodies,
i.e. its restriction to each finite dimensional section of the cone is a homogeneous polynomial of degree $n$. In other words,
for any collection of convex bodies $\Delta_1, \ldots, \Delta_r$, the function: 
$$P_{\Delta_1, \ldots, \Delta_r}(\lambda_1, \ldots, \lambda_r)
= \vol(\lambda_1\Delta_1 + \cdots + \lambda_r\Delta_r),$$
is a homogeneous polynomial of degree $n$ in $\lambda_1, \ldots, \lambda_r$. 
By definition the {\it mixed volume}  $V(\Delta_1,\dots,\Delta_n)$ of
an $n$-tuple $(\Delta_1,\dots,\Delta_n)$ of convex bodies
is the coefficient of the monomial $\lambda_1 \cdots \lambda_n$ in the polynomial
$P_{\Delta_1, \ldots, \Delta_n}(\lambda_1, \ldots, \lambda_n)$ divided by $n!$.
This definition implies that mixed volume is the {\it polarization} of the volume polynomial,
that is, it is the unique function on the $n$-tuples of convex bodies satisfying the
following:
\begin{itemize}
\item[(i)] (Symmetry) $V$ is symmetric with respect to permuting the bodies $\Delta_1, \ldots, \Delta_n$.
\item[(ii)] (Multi-linearity) It is {linear} in each argument with respect to the Minkowski sum. {The linearity} in first argument
means that for convex bodies $\Delta_1'$, $\Delta_1'',
\Delta_2,\dots,\Delta_n$, and real numbers $\lambda', \lambda'' \geq 0$ we have:
$$ V(\lambda'\Delta_1'+\lambda''\Delta_1'', \Delta_2, \dots, \Delta_n)=\lambda'V(\Delta_1', \Delta_2, \dots,
\Delta_n) + \lambda''V(\Delta_1'', \Delta_2, \dots, \Delta_n).$$
\item[(iii)] (Relation with volume) On the diagonal it coincides with the volume, i.e. if
$\Delta_1 =\cdots=\Delta_n=\Delta$, then $V(\Delta_1,\ldots,
\Delta_n)=\vol(\Delta)$.
\end{itemize}

The following inequality attributed to Alexandrov and Fenchel is important and very
useful in convex geometry (see \cite{BZ}):
\begin{Th}[Alexandrov-Fenchel] \label{thm-alexandrov-fenchel}
Let $\Delta_1, \ldots, \Delta_n$ be convex bodies
in $\r^n$. Then
$$ V(\Delta_1, \Delta_1, \Delta_3, \ldots, \Delta_n)
V(\Delta_2, \Delta_2, \Delta_3, \ldots, \Delta_n) \leq V(\Delta_1, \Delta_2, \ldots, \Delta_n)^2.$$
\end{Th}

In dimension $2$, this inequality is elementary. We call it the {\it generalized isoperimetric inequality}, because when
$\Delta_2$ is the unit ball it coincides with the classical isoperimetric inequality.
The celebrated {\it Brunn-Minkowski inequality} concerns volume of
convex bodies in $\r^n$. {It is an easy corollary of the Alexandrov-Fenchel
inequality. (For $n=2$ it is equivalent to the Alexandrov-Fenchel inequality.)}

\begin{Th}[Brunn-Minkowski] \label{th-Brunn-Mink}
Let $\Delta_1$, $\Delta_2$ be convex bodies in $\r^n$. Then
$$\vol(\Delta_1)^{1/n} + \vol(\Delta_2)^{1/n}\leq
\vol(\Delta_1+\Delta_2)^{1/n}.$$
\end{Th}

\section{Mixed covolume of convex regions} \label{sec-convex-diag}
Let $\C$ be a strongly convex closed $n$-dimensional cone in $\r^n$ with apex at the origin. (A convex cone is strongly convex 
if it does not contain any lines through the origin.) We are interested in closed convex subsets of $\C$ which have 
bounded complement. 

\begin{Def} \label{def-convex-region}
We call a closed convex subset $\Gamma \subset \C$ a {\it $\C$-convex region} (or simply a convex region when the cone $\C$ is understood from the context) if for any $x \in \Gamma$ and $y \in \C$ we have $x + y \in \Gamma$. Moreover, we say that a convex region $\Gamma$ is {\it cobounded} if the complement 
$\C \setminus \Gamma$ is bounded. In this case the volume of $\C \setminus \Gamma$ is finite which we call the {\it covolume of $\Gamma$} and denote it by $\covol(\Gamma)$. One also refers to $\C \setminus \Gamma$ as a {\it $\C$-coconvex body}.
\end{Def}

The collection of $\C$-convex regions (respectively cobounded regions) is closed under the Minkowski sum and multiplication by positive scalars. 
Similar to the volume of convex bodies, one proves that the covolume of convex regions is a homogeneous polynomial \cite{Askold-Vladlen}. More precisely:
\begin{Th} \label{th-covol-polynomial}
Let $\Gamma_1, \ldots, \Gamma_r$ be cobounded $\C$-convex regions in the cone $\C$. Then the function
$$P_{\Gamma_1, \ldots, \Gamma_r}(\lambda_1, \ldots, \lambda_r) = \covol(\lambda_1\Gamma_1 + \cdots + \lambda_r\Gamma_r),$$ 
is a homogeneous polynomial of degree $n$ in the $\lambda_i$. 
\end{Th}

As in the case of convex bodies, one uses the above theorem to define mixed covolume of cobounded regions. By definition the {\it mixed covolume}  $CV(\Gamma_1, \ldots, \Gamma_n)$ of
an $n$-tuple $(\Gamma_1,\dots,\Gamma_n)$ of cobounded convex regions
is the coefficient of the monomial $\lambda_1 \cdots \lambda_n$ in the polynomial 
$P_{\Gamma_1,\ldots, \Gamma_n}(\lambda_1, \ldots, \lambda_n)$ divided by $n!$.
That is, mixed covolume is the unique function on the $n$-tuples of cobounded regions satisfying the
following:
\begin{itemize}
\item[(i)] (Symmetry) $CV$ is symmetric with respect to permuting the regions $\Gamma_1, \ldots, \Gamma_n$.
\item[(ii)] (Multi-linearity) It is linear in each argument with respect to the Minkowski sum. 
\item[(iii)] (Relation with covolume) For any cobounded region $\Gamma \subset \C$: 
$$CV(\Gamma, \ldots, \Gamma)=\covol(\Gamma).$$
\end{itemize}

The mixed covolume satisfies an Alexandrov-Fenchel inequality \cite{Askold-Vladlen}. Note that 
the inequality is reversed compared to the Alexandrov-Fenchel for mixed volumes of convex bodies.
\begin{Th}[Alexandrov-Fenchel for mixed covolume] \label{thm-alexandrov-fenchel-covolume}
Let $\Gamma_1, \ldots, \Gamma_n$ be cobounded $\C$-convex regions. Then:
$$ CV(\Gamma_1, \Gamma_1, \Gamma_3, \ldots, \Gamma_n)
CV(\Gamma_2, \Gamma_2, \Gamma_3, \ldots, \Gamma_n) \geq CV(\Gamma_1, \Gamma_2, \Gamma_3, \ldots, \Gamma_n)^2.$$
\end{Th} 

The (reversed) Alexandrov-Fenchel inequality implies a (reversed) Brunn-Minkowski inequality. 
{ (For $n=2$ it is equivalent to the Alexandrov-Fenchel inequality.)}
\begin{Cor}[Brunn-Minkowski for covolume] \label{cor-Brunn-Mink-covol}
Let $\Gamma_1$, $\Gamma_2$ be cobounded $\C$-convex regions. Then:
$$\covol(\Gamma_1)^{1/n} + \covol(\Gamma_2)^{1/n} \geq
\covol(\Gamma_1 + \Gamma_2)^{1/n}.$$
\end{Cor}

\section{Semigroups of integral points} \label{sec-semigp-int}
In this section we recall some general facts from \cite{KKh-Annals} about the asymptotic behavior of semigroups of integral points. 
Let $S \subset \z^n \times \z_{\geq 0}$ be an additive semigroup. 
Let $\pi: \r^n \times \r \to \r$ denote the projection onto the second factor, and let 
$S_k = S \cap \pi^{-1}(k)$ be the set of points in $S$ at level $k$. 
For simplicity, assume $S_1 \neq \emptyset$ and that $S$ generates the whole lattice $\z^{n+1}$.
Define the function $H_S$ by:
$$H_S(k) = \#S_k.$$ We call $H_S$ the {\it Hilbert function of the semigroup $S$}. We wish to 
describe the asymptotic behavior of $H_S$ as $k \to \infty$.

Let $C(S)$ be the closure of the convex hull of $S \cup \{0\}$, that is, the smallest closed
convex cone (with apex at the origin) containing $S$. 
We call the projection of the convex set $C(S) \cap \pi^{-1}(1)$ to $\r^n$ (under the
projection onto the first factor $(x, 1) \mapsto x$),
the {\it Newton-Okounkov convex set of the semigroup $S$} and denote it by $\Delta(S)$.
In other words,
$$\Delta(S) = \overline{\bigcup_{k>0} \{x/k \mid (x, k) \in S_k\}}.$$
{ (From the fact that $S$ is a semigroup one can show that $\Delta(S)$ is a convex set.)}
If $C(S) \cap \pi^{-1}(0) = \{0\}$ then $\Delta(S)$ is compact and hence a convex body.

The Newton-Okounkov convex set $\Delta(S)$ is responsible for the
asymptotic behavior of the Hilbert function of $S$ (see \cite[Corollary 1.16]{KKh-Annals}):
\begin{Th} \label{th-semigp-NO}
The limit
$$\lim_{k \to \infty} \frac{H_S(k)}{k^n},$$ exists and is equal to 
$\vol(\Delta(S))$.
\end{Th}

\section{Primary sequences in a semigroup and convex regions} \label{sec-semigp-ideal}
In this section we discuss the notion of a primary graded sequence of subsets in a semigroup 
and describe its asymptotic behavior using Theorem \ref{th-semigp-NO}. 
In section \ref{sec-valuation-ideal} we will employ this to describe the asymptotic behavior of the Hilbert-Samuel function of a graded
sequence of $\m$-primary ideals in a local domain.
 

Let $\S \subset \z^n$ be an additive semigroup containing the origin. 
Without loss of generality we assume that $\S$ generates the whole $\z^n$.
Let as above $\C = C(\S)$ denote the cone of $\S$ i.e. 
the closure of convex hull of $\S = \S \cup \{0\}$. We also assume that $\C$ is a strongly convex cone, i.e. 
it does not contain any lines through the origin.

For two subsets $\I, \J \subset \S$, the sum $\I + \J$ is the set $\{x+y \mid x \in \I,~ y \in \J\}$.
For any integer $k > 0$, by the product $k*\I$ we mean $\I + \cdots + \I$ ($k$ times).



\begin{Def} \label{def-graded-seq-semigp-ideal}
A {\it graded sequence of subsets} in $\S$  is a sequence $\I_\bullet = 
(\I_1, \I_2, \ldots)$ of subsets such that for any $k, m > 0$ we have 
$\I_k + \I_m \subset \I_{k+m}$. 
\end{Def}

\begin{Ex} \label{ex-powers-of-semigp-ideal}
Let $\I \subset \S$. Then the sequence $\I_\bullet$ defined by $\I_k = k * \I$ is clearly 
a graded sequence of subsets. 
\end{Ex}

Let $\I'_\bullet$, $\I''_\bullet$ be graded sequences of subsets. Then the sequence 
$\I_\bullet = \I'_\bullet + \I''_\bullet$ defined by $$\I_k = \I'_k + \I''_k,$$
is also a graded sequence of subsets which we call the {\it sum of sequences 
$\I'_\bullet$ and $\I''_\bullet$}.

Let $\ell: \r^n \to \r$ be a linear function. For any $a \in \r$ let $\ell_{\geq a}$ (respectively $\ell_{> a}$) denote the half-space $\{ x \mid \ell(x) \geq a\}$ 
(respectively $\{ x \mid \ell(x) > a\}$), and similarly for $\ell_{\leq a}$ and $\ell_{< a}$.
By assumption the cone $\C = C(S)$ associated to the semigroup $\S$ is strongly convex. Thus we can find a linear 
function $\ell$ such that 
the cone $\C = C(\S)$ lies in $\ell_{\geq 0}$ and it intersects the hyperplane $\ell^{-1}(0)$ only at the origin. Let us fix such a linear function $\ell$.

We will be interested in graded sequences of subsets $\I_\bullet$ satisfying the following condition:
\begin{Def} \label{def-primary-sequence}
{We say that a graded sequence of subsets $\I_\bullet$ is {\it primary} if 
there exists an integer $t_0 > 0$ such that for any integer $k >0$ 
we have: }
\begin{equation} \label{equ-primary-semigroup-ideal}
\I_k \cap \ell_{\geq kt_0} = \S \cap \ell_{\geq kt_0}.
\end{equation}
\end{Def}

\begin{Rem}
One verifies that if $\ell'$ is another linear function such that $\C$ lies in $\ell'_{\geq 0}$ and it intersects the hyperplane $\ell'^{-1}(0)$ only at the origin then it automatically satisfies \eqref{equ-primary-semigroup-ideal} with perhaps a different constant $t'_0 > 0$. Hence the condition of being a primary graded sequence does not depend on the linear function $\ell$. Nevertheless when we refer to a primary graded sequence $\I_\bullet$, choices of a linear function $\ell$ and an integer $t_0 > 0$ are implied.
\end{Rem}

\begin{Prop} \label{prop-I_k-cofinite}
Let $\I_\bullet$ be a primary graded sequence. 
Then for all $k > 0$, the set $\S \setminus \I_k$ is finite.
\end{Prop}
\begin{proof}
Since $\C$ intersects $\ell^{-1}(0)$ only at the origin it follows that for any $k > 0$ the set $\C \cap \ell_{< kt_0}$ is bounded which implies that
$\S \cap \ell_{< kt_0}$ is finite. But by \eqref{equ-primary-semigroup-ideal}, $\S \setminus \I_k \subset \S \cap \ell_{<kt_0}$ and hence is finite.
\end{proof}

\begin{Def} \label{def-Hilbert-Samuel-semigroup}
Let $\I_\bullet$ be a primary graded sequence. Define the function $H_{\I_\bullet}$ by: $$H_{\I_\bullet}(k) = \#(\S \setminus \I_k).$$ 
(Note that by Proposition \ref{prop-I_k-cofinite} this number is finite for all $k>0$.) We call it the {\it Hilbert-Samuel function of $\I_\bullet$}.
\end{Def}

To a primary graded sequence of subsets $\I_\bullet$ we can associate a $\C$-convex region
$\Gamma(\I_\bullet)$ (see Definition \ref{def-convex-region}). This convex set encodes information about the 
asymptotic behavior of the Hilbert-Samuel function of $\I_\bullet$.

\begin{Def} \label{def-Gamma-I-semigroup}
Let $\I_\bullet$ be a primary graded sequence of subsets. Define the convex set $\Gamma(\I_\bullet)$ by
$$\Gamma(\I_\bullet) = \overline{\bigcup_{k>0} \{x/k \mid x \in \I_k\}}.$$
(One can show that $\Gamma(\I_\bullet)$ is an unbounded convex set in $\C$.)
\end{Def}

\begin{Prop}
Let $\I_\bullet$ be a primary graded sequence. 
Then $\Gamma = \Gamma(\I_\bullet)$ is a $\C$-convex region in the cone 
$\C$, i.e. for any $x \in \Gamma$, $x+\C \subset \Gamma$. 
Moreover, the region $\Gamma$
is cobounded i.e. $\C \setminus \Gamma$ is bounded.  
\end{Prop}
\begin{proof}
{ Let $\ell$ and $t_0$ be as in Definition \ref{def-primary-sequence}.
From the definitions it follows that the region $\Gamma$ contains $\C \cap \ell_{\geq t_0}$. Thus 
$(\C \setminus \Gamma) \subset (\C \cap \ell_{< t_0})$ and hence is bounded. Next let $x \in \Gamma$. Since $x+\C \subset \C$, 
$\Gamma$ contains the set $(x+\C) \cap \ell_{\geq t_0}$. But the convex hull of $x$ and $(x+\C) \cap \ell_{\geq t_0}$ is $x+\C$. Thus 
$x+\C \subset \Gamma$ because $\Gamma$ is convex.}
\end{proof}

{The following is an important example of a primary graded sequence in a semigroup $\S$. 
\begin{Prop} \label{prop-polyhedral-primary-semigroup-ideal}
Let $\C$ be an $n$-dimensional strongly convex rational polyhedral cone in $\r^n$ and let $\S = \C \cap \z^n$. 
Also let $\I \subset \S$ be a subset such that 
$\S \setminus \I$ is finite. Then the sequence $\I_\bullet$ defined by $\I_k := k * \I$ is a primary graded sequence and $\Gamma(\I_\bullet) = \conv(\I)$.
\end{Prop}
\begin{proof}
{Since $\S \setminus \I$ is finite there exists $t_1>0$ such that $\S \cap \ell_{\geq t_1} \subset \I$. Put $\M_1 = \S \cap \ell_{\geq t_1}$. Because 
$\C$ is a rational polyhedral cone, $\M_1$ is a finitely generated semigroup. Let $v_1, \ldots, v_s$ be semigroup generators for $\M_1$. 
Let $t_0 > 0$ be bigger than all the $\ell(v_i)$. For $k>0$ take $x \in \S \cap \ell_{\geq kt_0} \subset \M_1$. Then $x = \sum_{i=1}^s c_i v_i$ for 
$c_i \in \z_{\geq 0}$. Thus $kt_0 \leq \ell(x) = \sum_i c_i \ell(v_i) \leq (\sum_i c_i)t_0$. This implies that $k \leq \sum_i c_i$ and hence 
$(\sum_i c_i) * \M_1 \subset k * \M_1$. It follows that $x \in k*\M_1$. That is, 
$\S \cap \ell_{\geq kt_0} = (k * \M_1) \cap  \ell_{\geq kt_0} \subset (k * \I) \cap \ell_{\geq kt_0}$ and hence $\S \cap \ell_{\geq kt_0} = (k*\I) \cap \ell_{\geq kt_0}$ 
as required. The assertion $\Gamma(\I_\bullet) = \conv(\I_\bullet)$ follows from the observation that $\conv(k * \I) = k ~ \conv(\I)$.}
\end{proof}

The following is our main result about the asymptotic behavior of a primary graded sequence.
\begin{Th} \label{th-vol-Gamma-semigroup}
Let $\I_\bullet$ be a primary graded sequence. 
Then $$\lim_{k \to \infty} \frac{H_{\I_\bullet}(k)}{k^n}$$ exists and is equal 
to $\covol(\Gamma(\I_\bullet))$.
\end{Th}

\begin{proof}
Let $t_0 > 0$ be as in Definition \ref{def-primary-sequence}. Then for all  $k>0$ we have
$\S \cap \ell_{\geq kt_0} = \I_k \cap \ell_{\geq kt_0}.$
{ Moreover take $t_0$ to be large enough so that the finite set $\I_1 \cap \ell_{< t_0}$ 
generates the lattice $\z^n$ (this is possible because $\S$ and hence $\I_1$ generate $\z^n$).}
Consider 
$$\tilde{S} = \{(x, k) \mid x \in \I_k \cap \ell_{< kt_0} \}.$$
$$\tilde{T} = \{(x, k) \mid x \in \S \cap \ell_{< kt_0} \}.$$
$\tilde{S}$ and $\tilde{T}$ are semigroups in $\z^n \times \z_{\geq 0}$ and we have $\tilde{S} \subset \tilde{T}$. 
From the definition it follows that both of the groups generated by $\tilde{S}$ and $\tilde{T}$ are $\z^{n+1}$. 
Also the Newton-Okounkov bodies of $\tilde{S}$ and $\tilde{T}$ are: 
$$\Delta(\tilde{S}) = \Gamma(\I_\bullet) \cap \Delta(t_0),$$ 
$$\Delta(\tilde{T}) = \Delta(t_0),$$
where $\Delta(t_0) = \C \cap \ell_{\leq t_0}$. 
Since $\S \cap \ell_{\geq kt_0} = \I_k \cap \ell_{\geq kt_0}$, we have:
$$\S \setminus \I_k = \tilde{T}_k \setminus \tilde{S}_k,$$
Here as usual $\tilde{S}_k = \{(x, k) \mid (x, k) \in \tilde{S}\}$ (respectively $\tilde{T}_k$) denotes the set of points in $\tilde{S}$ (respectively $\tilde{T}$) at level $k$.
Hence $$H_{\I_\bullet}(k) = \#\tilde{T}_k - \#\tilde{S}_k.$$
By Theorem \ref{th-semigp-NO} we have:
$$\lim_{k \to \infty} \frac{\#\tilde{S}_k}{k^n} = \vol(\Delta(\tilde{S})),$$
$$\lim_{k \to \infty} \frac{\#\tilde{T}_k}{k^n} = \vol(\Delta(t_0)).$$
Thus $$\lim_{k \to \infty} \frac{\#(\S \setminus \I_k)}{k^n} = \vol(\Delta(t_0)) - \vol(\Delta(\tilde{S})).$$
On the other hand, we have:
$$\Delta(t_0) \setminus \Delta(\tilde{S}) = \C \setminus \Gamma(\I_\bullet),$$ and hence  
$\vol(\Delta(t_0)) - \vol(\Delta(\tilde{S})) = \covol(\Gamma(\I_\bullet)).$ This finishes the proof.
\end{proof}

\begin{Def} \label{def-multiplicity-semigp-ideal}
For a primary graded sequence $\I_\bullet$ we denote $n!~\lim_{k \to \infty} H_{\I_\bullet}(k) / k^n$ by
$e(\I_\bullet)$. Motivated by commutative algebra, we call it the {\it multiplicity of $\I_\bullet$}. 
We have just proved that $e(\I_\bullet) = n!~\covol(\Gamma(I_\bullet))$. Note that it is possible for a primary graded
sequence to have multiplicity equal to zero. 
\end{Def}

The following additivity property is straightforward from definition.
\begin{Prop} \label{prop-additivity-semigroup-ideal}
Let $\I_\bullet'$, $\I_\bullet''$ be primary graded sequences. We have:
$$\Gamma(\I_\bullet') + \Gamma(\I_\bullet'') = \Gamma(\I_\bullet' + \I_\bullet'').$$
\end{Prop}
\begin{proof}
From definition it is clear that $\Gamma(\I_\bullet'+\I_\bullet'') \subset \Gamma(\I_\bullet') + \Gamma(\I_\bullet'')$. We need to show the reverse inclusion. 
Let $\I_\bullet$ denote $\I_\bullet' + \I_\bullet''$.
For $k, m >0$ take $x \in \I_k'$ and $y \in \I_m''$. Then $(x/k) + (y/m) = (m x + ky)/km \in (1/km) \I_{km}$. This shows that $(x/k) + (y/m) \in \Gamma(\I_\bullet)$
which finishes the proof.
\end{proof}

Let $\I_{1, \bullet}, \ldots, \I_{n, \bullet}$ be $n$ primary graded sequences. 
Define the function $P_{\I_{1, \bullet}, \ldots, \I_{n, \bullet}}: \n^n \to \n$ by: $$P_{\I_{1, \bullet}, \ldots, \I_{n, \bullet}}(k_1, \ldots, k_n) = 
e(k_1*\I_{1, \bullet} + \cdots + k_n*\I_{n, \bullet}).$$
\begin{Th} \label{th-mixed-multi-poly-semigp}
The function $P_{\I_{1, \bullet}, \ldots, \I_{n, \bullet}}$ is a homogeneous polynomial of degree $n$ in $k_1, \ldots, k_n$.
\end{Th}
\begin{proof}
Follows immediately from Proposition \ref{prop-additivity-semigroup-ideal}, 
Theorem \ref{th-vol-Gamma-semigroup} and Theorem \ref{th-covol-polynomial}.
\end{proof}

\begin{Def} \label{def-mixed-multi-semigroup}
{Let $\I_{1, \bullet}, \ldots, \I_{n, \bullet}$ be primary graded sequences.
Define the {\it mixed multiplicity} $e(\I_{1, \bullet}, \ldots, \I_{n, \bullet})$ to be the coefficient of $k_1 \cdots k_n$ in 
the polynomial $P_{\I_{1, \bullet}, \ldots, \I_{n, \bullet}}$ divided by $n!$.}
\end{Def}



From Theorem \ref{th-vol-Gamma-semigroup} and Proposition \ref{prop-additivity-semigroup-ideal} we have the 
following corollary:
\begin{Cor}
$$e(\I_{1, \bullet}, \ldots, \I_{n, \bullet}) = n!~CV(\Gamma(\I_{1, \bullet}), \ldots, \Gamma(\I_{n, \bullet})),$$
where as before $CV$ denotes the mixed covolume of cobounded regions.
\end{Cor}

From Theorem \ref{thm-alexandrov-fenchel-covolume} and Corollary \ref{cor-Brunn-Mink-covol} we then obtain:

\begin{Cor}[Alexandrov-Fenchel inequality for mixed multiplicity in semigroups] \label{cor-AF-semigroup}
For primary graded sequences $\I_{1, \bullet}, \ldots, \I_{n, \bullet}$ in the semigroup $\S$ we have:
$$ e(\I_{1, \bullet}, \I_{1, \bullet}, \I_{3, \bullet}, \ldots, \I_{n, \bullet}) e(\I_{2, \bullet}, \I_{2, \bullet}, \I_{3, \bullet}, \ldots, \I_{n, \bullet}) 
\geq e(\I_{1, \bullet}, \I_{2, \bullet}, \I_{3, \bullet}, \ldots, \I_{n, \bullet})^2.$$
\end{Cor}

\begin{Cor}[Brunn-Minkowski inequality for multiplicities in semigroups]
{ Let $\I_\bullet$, $\J_\bullet$ be primary graded sequences in the semigroup $\S$. We have: 
$$e(\I_\bullet)^{1/n} + e(\J_\bullet)^{1/n} \geq e(\I_\bullet + \J_\bullet)^{1/n}.$$}
\end{Cor}

\section{Multiplicities of ideals and subspaces in local rings} \label{sec-multi-ideal}
Let $R$ be a Noetherian local domain of Krull dimension $n$ over a field $\k$, and with maximal ideal $\m$. 
We also assume that the residue field $R/\m$ is $\k$.

\begin{Ex} \label{ex-local-ring-of-subvariety}
Let $X$ be an irreducible variety of dimension $n$ over $\k$, and let $p$ be a point in $X$. Then the 
local ring $R = \mathcal{O}_{X, p}$ (consisting
of rational functions on $X$ which are regular in a neighborhood of $p$) is a Noetherian local domain of Krull dimension $n$ over $\k$.
The ideal $\m$ consists of functions which vanish at $p$.
\end{Ex}

{

If $\a, \b \subset R$ are two $\k$-subspaces then by $\a\b$ we denote the $\k$-span of all the $xy$ where $x \in \a$ and $y \in \b$.
Note that if $\a, \b$ are ideals in $R$ then $\a\b$ coincides with the product of $\a$ and $\b$ as ideals.

\begin{Def}\label{def-graded-seq-subspace}
\begin{itemize}
\item[(i)] A $\k$-subspace $\a$ in $R$ is called {\it $\m$-primary} if it contains a power of the maximal ideal $\m$.
\item[(ii)] A {\it graded sequence of subspaces} is a sequence $\a_\bullet = (\a_1, \a_2, \ldots)$ of $\k$-subspaces in $R$ such that for all $k, m>0$ we have $\a_k \a_m \subset \a_{k+m}$. We call $\a_\bullet$ an {\it $\m$-primary sequence} if moreover $\a_1$ is $\m$-primary. It then follows that every $\a_k$ is $\m$-primary and hence $\dim_{\k}(R/\a_k)$ is finite. (If each $\a_k$ is an $\m$-primary ideal in $R$ we call 
$\a_\bullet$ an {\it $\m$-primary graded sequence of ideals}.) 
\end{itemize}
\end{Def}
}

When $\k$ is algebraically closed, 
an ideal $\a$ in $R = \mathcal{O}_{X, p}$ is $\m$-primary if the subvariety it defines around $p$ coincides with the single 
point $p$ itself.

\begin{Ex} \label{ex-powers-of-semigp-ideal}
Let $\a$ be an $\m$-primary subspace. Then the sequence $\a_\bullet$ defined by $\a_k = \a^k$ is an $\m$-primary graded sequence of subspaces.  
\end{Ex}

Let $\a_\bullet$, $\b_\bullet$ be $\m$-primary graded sequences of subspaces. Then the sequence 
$\mathfrak{c}_\bullet = \a_\bullet \b_\bullet$ defined by $$\mathfrak{c}_k = \a_k \b_k,$$
is also an $\m$-primary graded sequence of subspaces which we call the {\it product of $\a_\bullet$ and $\b_\bullet$}.

\begin{Def} \label{def-Hilbert-Samuel-semigroup}
Let $\a_\bullet$ be an $\m$-primary graded sequence of subspaces.
Define the function $H_{\a_\bullet}$ by: $$H_{\a_\bullet}(k) = \dim_\k(R / \a_k).$$ We call it the 
{\it Hilbert-Samuel function of $\a_\bullet$}.
The {\it Hilbert-Samuel function} $H_\a(k)$ of an $\m$-primary subspace $\a$ is the Hilbert-Samuel function of the 
sequence $\a_\bullet = (\a, \a^2, \ldots)$. That is, $H_\a(k) = \dim_\k(R / \a^k)$.
\end{Def}

\begin{Rem} \label{rem-Hilbert-Samuel-poly}
For an $\m$-primary ideal $\a$ it is well-known that, for sufficiently large values of $k$, the Hilbert-Samuel function $H_\a$ coincides with a polynomial of degree $n$ 
called the {\it Hilbert-Samuel polynomial of $\a$}
(\cite{SZ}).
\end{Rem}

\begin{Def} \label{def-multiplicity}
Let $\a_\bullet$ be an $\m$-primary graded sequence of subspaces.
We define the {\it multiplicity $e(\a_\bullet)$} to be:
$$e(\a_\bullet) = n!~\lim_{k \to \infty} \frac{H_{\a_\bullet}(k)}{k^n}.$$
(It is not a priori clear that the limit exists.)
The multiplicity $e(\a)$ of an $\m$-primary ideal $\a$ is the multiplicity of its associated sequence $(\a, \a^2, \ldots)$. 
That is:
$$e(\a) = n!~\lim_{k \to \infty} \frac{H_{\a}(k)}{k^n}.$$
{ (Note that by Remark \ref{rem-Hilbert-Samuel-poly} the limit exists in this case.)}
\end{Def}

{
The notion of multiplicity comes from the following basic example:
\begin{Ex} \label{ex-meaning-multiplicity}
Let $\a$ be an $\m$-primary subspace in the 
local ring $R = \mathcal{O}_{X, p}$ of a { point} $p$ in an irreducible variety $X$ over an algebraically closed filed $\k$.
Let $f_1, \ldots, f_n$ be generic elements in $\a$.
Then the multiplicity $e(\a)$ is equal to the intersection multiplicity at $p$ of the hypersurfaces $H_i = \{ x \mid f_i(x) = 0\}$, 
$i=1, \ldots, n$.
\end{Ex}
}

In Section \ref{sec-valuation-ideal} we use the material in Section \ref{sec-semigp-ideal} to give a formula for $e(\a_\bullet)$ 
in terms of covolume of a convex region. 

One can also define the notion of mixed multiplicity for $\m$-primary ideals 
as the polarization of the Hilbert-Samuel multiplicity 
$e(\a)$, i.e. it is the unique function $e(\a_1, \ldots, \a_n)$ which is invariant under permuting the arguments, is multi-additive with respect to product, and for any $\m$-primary ideal 
$\a$ the mixed multiplicity $e(\a, \ldots, \a)$ coincides with $e(\a)$. 
In fact one can show that in the above definition of mixed multiplicity 
the $\a_i$ need not be ideals and it suffices for them to be $\m$-primary subspaces.

Similar to multiplicity we have the following geometric meaning for the notion of mixed multiplicity when $R = \mathcal{O}_{X, p}$ is the local ring of a  point $p$ on an $n$-dimensional algebraic variety $X$.
Take $\m$-primary subspaces $\a_1, \ldots, \a_n$ in $R$. The mixed multiplicity $e(\a_1, \ldots, \a_n)$ is equal to the intersection multiplicity, at the origin, of the hypersurfaces 
$H_i = \{ x \mid f_i(x) = 0\}$, $i = 1, \ldots, n$, where each $f_i$ is a generic function from the $\a_i$. 
}

\section{Case of monomial ideals and Newton polyhedra} \label{sec-monomial-ideal}
In this section we discuss the case of monomial ideals. It is related to the classical notion of Newton polyhedron of a power series in $n$ variables. We will see that our Theorem \ref{th-vol-Gamma-semigroup} in this case immediately recovers (and generalizes) the local version of celebrated theorem of Bernstein-Kushnirenko (\cite{Kushnirenko} and \cite[Section 12.7]{AVG}).

Let $R$ be the local ring of an affine toric variety at its torus fixed point. The algebra $R$ can be realized as follows: Let $\C \subset \r^n$ be 
an  $n$-dimensional strongly convex rational polyhedral cone with apex at the origin, that is, $\C$ is an $n$-dimensional convex cone generated by a finite number of rational vectors and it does not contain any lines through the origin. Consider the semigroup algebra over $\k$ of the semigroup of integral points $\S = \C \cap \z^n$. In other words, consider the algebra of Laurent polynomials consisting of all the $f$ of the form
$f = \sum_{\alpha \in \C \cap \z^n} c_\alpha x^\alpha$, where we have used the shorthand notation $x = (x_1, \ldots, x_n)$, 
$\alpha = (a_1, \ldots, a_n)$ and $x^\alpha = x_1^{a_1} \cdots x_n^{a_n}$. Let $R$ be the localization of this Laurent polynomial algebra at the maximal ideal $\m$ generated by non-constant monomials. (Similarly instead of $R$ we can take its completion at the maximal ideal $\m$ which is an algebra of power series .)
 

\begin{Def} \label{def-Newton-polyhedron}
Let $\a$ be an $\m$-primary monomial ideal in $R$, that is, an $\m$-primary ideal generated by 
monomials. To $\a$ we can associate a subset $\I(\a) \subset \C \cap \z^n$ by
$$\I(\a) = \{\alpha \mid x^\alpha \in \a\}.$$
The convex hull $\Gamma(\a)$ 
of $\I(\a)$ is usually called the {\it Newton polyhedron} of the monomial ideal $\a$. 
It is a convex unbounded polyhedron in $\C$, moreover it is a 
$\C$-convex region. The {\it Newton diagram of $\a$} is the union of bounded faces of its Newton polyhedron.
\end{Def}

{
\begin{Rem} \label{rem-semigp-ideal-monomial-ideal}
It is easy to see that if $\a$ is an ideal in $R$ then $\I = \I(\a)$ is a semigroup ideal in $\S = \C \cap \z^n$, 
that is, if $x \in \I$ and $y \in \S$ then $x+y \in \I$. 
\end{Rem}
}

Let $\a$ be an $\m$-primary monomial ideal. Then for any $k>0$ we have $\I(\a^k) = k * \I(\a)$. It follows from Proposition 
\ref{prop-polyhedral-primary-semigroup-ideal} that $\I_\bullet$ defined by $\I_k = k * \I(\a)$ is a primary graded sequence in 
$\S = \C \cap \z^n$ and the convex region $\Gamma(\I_\bullet)$ associated to $\I_\bullet$ coincides with the 
Newton polyhedron $\Gamma(\a) = \conv(\I(\a))$ defined above. 

More generally, let $\a_\bullet$ be an $\m$-primary graded sequence of monomial ideals in $R$.
Associate a graded sequence $\I_\bullet = \I(\a_\bullet)$ in $\S$ to $\a_\bullet$ by: 
$$\I_k = \I(\a_k) = \{ \alpha \mid x^\alpha \in \a_k\}.$$
Clearly for any $k$ we have $k * \I(\a_1) = \I(\a_1^k) \subset \I(\a_k)$.
From the above we then conclude the following:
\begin{Prop} \label{prop-convex-region-monomial-ideal} 
The graded sequence $\I_\bullet = \I(\a_\bullet)$ is a primary graded sequence 
in the sense of Definition \ref{def-primary-sequence}. 
\end{Prop}

Let $\Gamma(\a_\bullet)$ denote the convex region associated to the primary graded sequence $\I_\bullet = \I(\a_\bullet)$ 
(Definition \ref{def-Gamma-I-semigroup}).
 We make the following important observation that $\a_\bullet \mapsto \Gamma(\a_\bullet)$ is additive with respect to 
the product of graded sequences of monomial ideals.
\begin{Prop} \label{prop-Newton-poly-additive}
Let $\a_\bullet, \b_\bullet$ be $\m$-primary graded sequences of monomial ideals in $R$. 
Then $\I(\a_\bullet \b_\bullet) = \I(\a_\bullet) + \I(\b_\bullet)$. It follows from Proposition 
\ref{prop-additivity-semigroup-ideal} that: $$\Gamma(\a_\bullet \b_\bullet) = \Gamma(\a_\bullet) + \Gamma(\b_\bullet).$$ 
\end{Prop}
  
From Proposition \ref{prop-convex-region-monomial-ideal}, Proposition \ref{prop-Newton-poly-additive} and 
Theorem \ref{th-vol-Gamma-semigroup} we readily obtain the following.


\begin{Th} \label{th-covol-monomial}
Let $\a_\bullet$ be an $\m$-primary graded sequence of monomial ideals in $R$. Then:
$$e(\a_\bullet) = n!~ \covol(\Gamma(\a_\bullet)).$$
In particular, if $\a$ is an $\m$-primary monomial ideal then:
$$e(\a) = n!~ \covol(\Gamma(\a)).$$ 
Here $\Gamma(\a)$ is the Newton polyhedron of $\a$ i.e. the convex hull of $\I(\a)$.
\end{Th}

\begin{Th} \label{th-mixed-multi-mixed-covol-monomial}
Let $\a_{1}, \ldots, \a_{n}$ be $\m$-primary monomial ideals in $R$. Then the mixed multiplicity 
$e(\a_{1}, \ldots, \a_{n})$ is given by: $$e(\a_{1}, \ldots, \a_{n}) = n!~CV(\Gamma(\a_{1}), \ldots, \Gamma(\a_{n})),$$
where as before $CV$ denotes the mixed covolume of cobounded convex regions. 
\end{Th}

\begin{Rem} \label{rem-mixed-multi-graded-seq-monomial}
Using Theorem \ref{th-mixed-multi-poly-semigp}
one can define the mixed multiplicity of $\m$-primary graded sequences of monomial ideals. Then 
Theorem \ref{th-mixed-multi-mixed-covol-monomial} can immediately be extended to mixed multiplicities of 
$\m$-primary graded sequences of monomial ideals. 
\end{Rem}

One knows that the mixed multiplicity of an $n$-tuple $(\a_1, \ldots, \a_n)$ of $\m$-primary subspaces in $R$ 
gives the intersection multiplicity, at the origin, of hypersurfacs $H_i = \{ x \mid f_i(x) = 0\}$, $i=1, \ldots, n$, where each $f_i$ is a 
generic element from the subspace $\a_i$. Theorem \ref{th-mixed-multi-mixed-covol-monomial} then 
gives the following corollary. 
\begin{Cor}[Local Bernstein-Kushnirenko theorem] \label{cor-local-BK-toric-sing}
Let $\a_1, \ldots, \a_n$ be $\m$-primary monomial ideals in $R$. Consider a system of equations 
$f_1(x) = \cdots = f_n(x) = 0$ where each $f_i$ is a generic element from $\a_i$. Then the 
intersection multiplicity at the origin of this system is equal to $n!~\covol(\Gamma(\a_1), \ldots, \Gamma(\a_n)).$
\end{Cor}

\begin{Rem}
\begin{itemize}
\item[(i)] When $R$ is the algebra of polynomials $\k[x_1, \ldots, x_n]_{(0)}$ localized at the 
origin (or the algebra of power series localized at the origin), i.e. the case corresponding to the local ring of a smooth affine toric variety, 
Corollary \ref{cor-local-BK-toric-sing} is the local version of the classical Bernstein-Kushnirenko theorem (\cite[Section 12.7]{AVG}).
\item[(ii)] As opposed to the proof above, the original proof of the Kushnirenko theorem is quite involved.
\item[(iii)] Corollary \ref{cor-local-BK-toric-sing} has been known to the second author since the early 90's (cf. \cite{Askold-finite-sums}), although as far as the authors know it has not been published.
\end{itemize}
\end{Rem}



\section{Main results} \label{sec-valuation-ideal}
Let $R$ be a domain over a field $\k$. Equip the group $\z^n$ with a total order respecting addition.

\begin{Def}[Valuation] \label{def-valuation}
A {\it valuation} $v: R \setminus\{0\} \to \z^n$ is a function satisfying:
\begin{enumerate}
\item For all $0 \neq f, g \in R$, $v(fg) = v(f) + v(g)$.
\item For all $0 \neq f,g \in R$ with $f+g \neq 0$ we have $v(f+g) \geq \min(v(f), v(g))$. (One then shows that 
when $v(f) \neq v(g)$, $v(f+g) = \min(v(f), v(g))$.)
\item For all $0 \neq \lambda \in \k$, $v(\lambda) = 0$.
\end{enumerate}
We say that $v$ has {\it one-dimensional leaves}
if whenever $v(f) = v(g)$, there exists $\lambda \in \k$ with
$v(g + \lambda f) > v(g)$.
\end{Def}

{ From definition $\S = v(R \setminus \{0\}) \cup \{0\}$ is an additive subsemigroup of $\z^n$ which we call the 
{\it value semigroup of $(R, v)$}. Any valuation on $R$ extends to the field of fractions $K$ of $R$ by defining $v(f/g) = v(f) - v(g)$. The set 
$R_v = \{0 \neq f \in K \mid v(f) \geq 0\} \cup \{0\}$ is a local subring of $K$ called the {\it valuation ring of $v$}. Also
$\m_v = \{0 \neq f \in K \mid v(f) > 0\} \cup \{0\}$ is the maximal ideal in $R_v$. The field $R_v / \m_v$ is called the 
{\it residue field of $v$}. One can see that $v$ has one-dimensional leaves if and only if the residue field of $v$ is $\k$.}

\begin{Def} \label{def-semigp-ideal-of-ideal}
For a subspace $\a$ in $R$ define
$\I = \I(\a) \subset \S$ by:
$$\I = \{ v(f) \mid f \in \a \setminus \{0\} \}.$$
Similarly, for a graded sequence of subspaces $\a_\bullet$ in $R$, define
$\I_\bullet = \I(\a_\bullet)$ by:
$$\I_k = \I(\a_k) = \{ v(f) \mid f \in \a_k \setminus \{0\} \}.$$
\end{Def}

For the rest of the paper we assume that $R$ is a Noetherian local domain of dimension $n$ such that $R$ is 
an algebra over a field $\k$ isomorphic to the residue field $R/\m$, where $\m$ is the maximal ideal of $R$. Moreover, we assume that
$R$ has a {good} valuation in the following sense:

{
\begin{Def} \label{def-good-valuation}
We say that a $\z^n$-valued valuation $v$ on $R$ with one-dimensional leaves is {\it good} if the following hold: 

\noindent{(i)} The value semigroup $\S = v(R \setminus \{0\}) \cup \{0\}$ generates the whole lattice $\z^n$, and 
its associated cone $C(\S)$ is a strongly convex cone (recall that $C(\S)$ is the closure of convex hull of $\S$). It implies that there is a linear function $\ell: \r^n \to \r$ such that $C(\S)$ lies in $\ell_{\geq 0}$ and intersects $\ell^{-1}(0)$ only at the origin.
 
\noindent{(ii)} We assume there exists $r_0 > 0$ and a linear function $\ell$ as above such that for any $f \in R$ 
if $\ell(v(f)) \geq kr_0$ for some $k>0$ then $f \in \m^k$.
  
The condition (ii) in particular implies that for any $k>0$ we have: 
$$\I(\m^k) \cap \ell_{\geq kr_0} = \S \cap \ell_{\geq kr_0}.$$
In other words, the sequence $\M_\bullet$ given by 
$\M_k = \I(\m^k)$ is a primary graded sequence in the value semigroup $\S$.
\end{Def}
}
The following is a generalization of Proposition \ref{prop-convex-region-monomial-ideal}.
\begin{Prop} \label{prop-good-val-primary-ideal}
Let $v$ be a good valuation on $R$. 
Let $\a_\bullet$ be an $\m$-primary graded sequence of subspaces in $R$. Then the associated graded sequence
$\I_\bullet = \I(\a_\bullet)$ is a primary graded sequence in the value semigroup $\S$ in the sense of 
Definition \ref{def-primary-sequence}. 
\end{Prop}
\begin{proof}
Let $m>0$ be such that $\m^m \subset \a_1$. Then 
for any $k>0$ we have $\m^{km} \subset \a_{k}$ which then implies that $\M_{km} \subset \I_k$. This proves the claim. 
\end{proof}

{
\begin{Ex} \label{ex-good-val-toric-local-ring}
As in Section \ref{sec-monomial-ideal} let $R$ be the local ring of an affine toric variety at its torus fixed point:
Take $\C \subset \r^n$ to be an $n$-dimensional strongly convex rational 
polyhedral cone with apex at the origin.
Consider the algebra of Laurent polynomials consisting of all the $f$ of the form
$f = \sum_{\alpha \in \C \cap \z^n} c_\alpha x^\alpha$. 
Then $R$ is the localization of this algebra at the maximal ideal generated by non-constant monomials.
Take a total order on $\z^n$ which respects addition and such that the semigroup
$\S = \C \cap \z^n$ is well-ordered. We also require that if $\ell(\alpha) > \ell(\beta)$ then $\alpha > \beta$, for any $\alpha, \beta \in \z^n$.
Such a total order can be constructed as follows: pick linear functions $\ell_2, \ldots, \ell_n$ on $\r^n$ 
such that $\ell, \ell_2, \ldots, \ell_n$ are linearly independent, and 
for each $i$ the cone $\C$ lies in $(\ell_{i})_{\geq 0}$.
Given $\alpha, \beta \in \z^n$ define $\alpha > \beta$ if $\ell(\alpha) > \ell(\beta)$, or $\ell(\alpha) = \ell(\beta)$ and $\ell_2(\alpha) > \ell_2(\beta)$, and so on.

Now one defines a (lowest term) valuation $v$ on the algebra $R$ with values in $\S = \C \cap \z^n$ as follows: 
For $f = \sum_{\alpha \in \S} c_\alpha x^\alpha$ put: $$v(f) = \min\{ \alpha \mid c_\alpha \neq 0\}.$$ 
Clearly $v$ extends to the field of fractions of Laurent polynomials and in particular to $R$. Similarly $v$ can be defined for 
formal power series and formal Laurent series.
It is easy to see that $v$ is a valuation with one-dimensional leaves on $R$. Let us show that it is moreover a good valuation. 
Take $0 \neq f \in R$. Without loss of generality we can take $f$ to be a Laurent series $f = \sum_{\alpha \in \S} c_\alpha x^\alpha$.
Applying Proposition \ref{prop-polyhedral-primary-semigroup-ideal} to the sequence $\I_\bullet$, where $\I_k = \{ \alpha \mid x^\alpha \in \m^k\}$, 
we know that there exists $r_0>0$ with the following property: if for some $\alpha \in \S$ we have 
$\ell(\alpha) \geq kr_0$ then $x^\alpha \in \m^k$.
On the other hand, if $\alpha \leq \beta$ then $\ell(\alpha) \leq \ell(\beta)$. Thus $\ell(\beta) \geq kr_0$ and hence $x^\beta \in \m^k$. It follows that 
if $\ell(v(f)) \geq kr_0$ then all the nonzero monomials in $f$ lie in $\m^k$ and hence $f \in \m^k$.
This proves the claim that $v$ is a good valuation on $R$. 
\end{Ex}
}
The arguments in Example \ref{ex-good-val-toric-local-ring} in particular show the following:
\begin{Th} \label{th-good-val-reg-local-ring}
If $R$ is a regular local ring then $R$ has a good valuation.
\end{Th}
\begin{proof}
The completion $\overline{R}$ of $R$ is isomorphic to an algebra of formal power series over the residue field $\k$. The above construction gives a 
good valuation on $\overline{R}$. One verifies that the restriction of this valuation to $R$ is still a good valuation. 
\end{proof}

More generally, one has:
\begin{Th} \label{th-good-val-S/R}
Suppose $R$ is an analytically irreducible local domain (i.e. the completion of $R$ has no zero divisors). Moreover, suppose that 
there exists a regular local ring $S$ containing $R$ such that $S$ is { essentially of finite type} over $R$, $R$ and $S$ have the same quotient field $\k$ 
and the residue field map $R/\m_R \to S/\m_S$ is an isomorphism. Then $R$ has a good valuation.
\end{Th}
\begin{proof}
By Theorem \ref{th-good-val-reg-local-ring}, $S$ has a good valuation.
By the linear Zariski subspace  theorem in \cite[Theorem 1]{Hubl} or \cite[Lemma 4.3]{Cutkosky1}
$v_{|R}$ is a good valuation for $R$ too.
\end{proof}

Using Theorem \ref{th-good-val-S/R} and as in \cite[Theorem 5.2]{Cutkosky1} we have the following:
\begin{Th} \label{th-good-val-analytically-irr-sing}
Let $R$ be an analytically irreducible local domain over $\k$. Then $R$ has a good valuation. 
\end{Th}

{
\begin{Prop}
Let $\I = \I(\a)$ be the subset of integral points associated to an $\m$-primary subspace 
$\a$ in $R$. Then we have: 
$$\dim_\k(R / \a) = \#( \S \setminus \I).$$
\end{Prop}
\begin{proof}
Take $m>0$ with $\m^m \subset \a$ and let $r_0$ and $\ell$ be as in Definition \ref{def-good-valuation}. If 
$\ell(v(f)) > mr_0$ then $f \in \m^{m} \subset \a$. Thus the set of valuations of elements in $R \setminus \a$ is bounded. In particular 
$\S \setminus \I$ is finite. 
Let $\{v_1, \ldots, v_r\}  = \S \setminus \I$.  Let $B = \{b_1, \ldots, b_r\} \subset R$ be such that $v(b_i) = v_i$, $i=1, \ldots, r$.
We claim that no linear combination of $b_1, \ldots, b_r$ lies in $\a$. By contradiction suppose 
$\sum_{i} c_i b_i = a \in \a$. Then $v(\sum_i c_i b_i)$ is equal to $v(b_j)$ for some $j$. This implies that $v(b_j)$ should lie in 
$\I$ which contradicts the choice of the $v_i$. Thus the image of $B$ in $R/\a$ is a linearly independent set. 
Among the set of elements in $R$ that are not in the span of $\a$ and $B$ take $f$ with maximum $v(f)$. If $v(f) = v(b)$ for some 
$b \in B$, then we can subtract a multiple of $b$ from $f$ getting an element $g$ with $v(g) > v(f)$ which contradicts the choice of
$f$. Similarly $v(f)$ can not lie in $\I$ otherwise we can subtract an element of $\a$ from $f$ to arrive at a similar contradiction. 
This shows that the set of images of elements of $B$ in $R/\a$ is a $\k$-vector space basis for $R/\a$ which proves the proposition.
\end{proof}
}

\begin{Cor}
Let $\a_\bullet$ be an $\m$-primary graded sequence of subspaces in $R$ and put $\I_\bullet = \I(\a_\bullet)$. 
We then have:
$$e(\a_\bullet) = e(\I_\bullet).$$
\end{Cor}

\begin{Def} \label{def-Gamma-I}
To the sequence of subspaces $\a_\bullet$ we associate a $\C$-convex region $\Gamma(\a_\bullet)$, which is the 
convex region $\Gamma(\I_\bullet)$ associated to the primary sequence $\I_\bullet = \I(\a_\bullet)$.
The convex region $\Gamma(\a_\bullet)$ depends on the choice of the valuation $v$. By definition the convex region $\Gamma(\a)$ 
associated to an $\m$-primary subspace $\a$ is the convex region associated to the sequence of subspaces $(\a, \a^2, \a^3, \ldots)$.
\end{Def}

\begin{Th} \label{th-multi-ideal-covol}
Let $\a_\bullet$ be an $\m$-primary graded sequence of subspaces in $R$. Then:
$$e(\a_\bullet) = n!~ \lim_{k \to \infty} \frac{H_{\a_\bullet}(k)}{k^n} = n!~ \covol(\Gamma(\a_\bullet)).$$
In particular, if $\a$ is an $\m$-primary ideal we have $e(\a) = n!~\covol(\Gamma(\a))$.
\end{Th}

The following superadditivity follows from Proposition \ref{prop-additivity-semigroup-ideal}.
Note that $\I(\a_\bullet) + \I(\b_\bullet) \subset \I(\a_\bullet \b_\bullet)$ (cf. Proposition \ref{prop-Newton-poly-additive}). 
\begin{Prop} \label{prop-superadd-ideal}
Let $\a_\bullet$, $\b_\bullet$ be two $\m$-primary graded sequences of subspaces in $R$. We have:
$$\Gamma(\a_\bullet) + \Gamma(\b_\bullet) \subset \Gamma(\a_\bullet \b_\bullet).$$ 
\end{Prop}

From Theorem \ref{th-multi-ideal-covol}, Proposition \ref{prop-superadd-ideal} and Corollary \ref{cor-Brunn-Mink-covol} we readily obtain:
\begin{Cor}(Brunn-Minkowski for multiplicities) \label{cor-Brunn-Mink-multi}
Let $\a_\bullet$, $\b_\bullet$ be two $\m$-primary graded sequences of subspaces in $R$. Then:
\begin{equation} \label{equ-Brunn-Mink-graded-seq-ideals}
e(\a_\bullet)^{1/n} + e(\b_\bullet)^{1/n} \geq e(\a_\bullet \b_\bullet)^{1/n}.
\end{equation}
\end{Cor}

\begin{Rem} \label{rem-BM-non-local-analytic-irreducible}
By Theorem \ref{th-good-val-analytically-irr-sing} and Corollary \ref{cor-Brunn-Mink-multi} we obtain the 
Brunn-Minkowski inequality \eqref{equ-Brunn-Mink-graded-seq-ideals} for   
an analytically irreducible local domain $R$. 
But in fact the assumption that $R$ is analytically irreducible is not necessary: Suppose $R$ is not necessarily
analytically irreducible. 
First by a reduction theorem the statement can be reduced to $\dim R = n = 2$. In dimension $2$, the inequality 
\eqref{equ-Brunn-Mink-graded-seq-ideals} 
implies that the mixed multiplicity of ideals $e(\cdot, \cdot)$, regarded as a bilinear form on the (multiplicative) semigroup of $\m$-primary 
graded sequences of ideals, is positive semidefinite restricted to each local analytic irreducible component. But the sum of positive semidefinite 
forms is again positive semidefinite which implies that Corollary \ref{cor-Brunn-Mink-multi} should hold for $R$ itself.
\end{Rem}

As another corollary of Theorem \ref{th-multi-ideal-covol} we can immediately obtain 
inequalities between the multiplicity of an $\m$-primary ideal, multiplicity of its associated initial ideal and 
its length. Let $R$ be a regular local ring of dimension $n$ with a good valuation (as in Example \ref{ex-good-val-toric-local-ring} and 
Theorem \ref{th-good-val-reg-local-ring}).
\begin{Cor}[Multiplicity of an ideal versus multiplicity of its initial ideal] \label{cor-Lech}
Let $\a$ be an $\m$-primary ideal in $R$ and let $\In(\a)$ denote the initial ideal of $\a$, that is, the monomial ideal in the polynomial 
algebra localized at the origin $\k[x_1, \ldots, x_n]_{(0)}$ corresponding to the semigroup ideal $\I(\a)$. We have:
$$e(\a) \leq e(\In(\a)) \leq n!~ \dim_\k(R / \a).$$
More generally, if $\In(\a^k)$ denote the monomial ideal in $\k[x_1, \ldots, x_n]_{(0)}$ corresponding to the semigroup ideal $\I(\a^k)$ then 
the sequence of numbers $$\frac{e(\In(\a^k))}{k^n}$$ is decreasing and converges to $e(\a)$ as $k \to \infty$. 
\end{Cor}
\begin{proof}
From definition one shows that $\Gamma(\In(\a))$ is the convex hull of $\I(\a)$ (see Theorem \ref{th-covol-monomial}). 
It easily follows that $$\I(\a) \subset \Gamma(\In(\a)) \subset \Gamma(\a).$$
We now notice that $\dim_\k(R / \a)$ is the number of integral points in $\S \setminus \I(\a)$ that in turn 
is bigger than or equal to the volume of $\R \setminus \Gamma(\In(\a))$ and hence the volume of $\R \setminus \Gamma(\a)$.
More generally, from the definition of $\Gamma(\a)$ we have an increasing sequence of convex regions:
$$\Gamma(\In(\a)) \subset (1/2)\Gamma(\In(\a^2)) \subset \cdots \subset \Gamma(\a) = \overline{\bigcup_{k=1}^\infty (1/k)\Gamma(\In(\a^k))}.$$   
Now from Theorem \ref{th-multi-ideal-covol} we have $e(\a) = n!~\covol(\Gamma(\a))$ and for each $k$, 
$e(\In(\a^k)) = n!~\covol(\Gamma(\In(\a^k)))$. This finishes the proof.
\end{proof}
 
The inequality $e(\a) \leq n!~\dim_\k(R / \a)$ is a special case of an inequality of Lech \cite[Theorem 3]{Lech}. See also Lemma 1.3
in \cite{FEM}.



\end{document}